\documentclass[11pt]{article}
\usepackage[dvips]{graphicx}
\usepackage{amssymb}
\usepackage{amsmath}
\usepackage{amsthm}
\usepackage{amsfonts}
\usepackage{amsthm,amscd}
\usepackage{amsbsy}
\usepackage{bm}
\usepackage{url} 
\usepackage{setspace}
\usepackage{float}
\usepackage{psfrag}
\usepackage{cite}
\restylefloat{figure}

\newtheorem{theorem}{Theorem}
\newtheorem{lemma}{Lemma}

\newtheorem{problem}{Problem}

\newtheorem{proposition}{Proposition}

\newtheorem{remark}{Remark}
\newtheorem{corollary}{Corollary}
\usepackage{setspace}

\def\be{\begin{equation}}
\def\ee{\end{equation}}
\def\ben{\begin{eqnarray}}
\def\een{\end{eqnarray}}
\newcommand{\bzero}{\mbox{$\bf 0$}}
\newcommand{\la}{\langle}
\newcommand{\ra}{\rangle}

\newcommand{\bn}{\mathbf n}
\newcommand{\BBR}{\mathbb R}

\newcommand{\bQ}{\mathbf{Q}}

\newcommand{\fb}{\mathbf{f}}

\newcommand{\bv}{\mathbf{v}}

\newcommand{\bV}{\mathbf{V}}
\newcommand{\balpha}{\mbox{\boldmath{$\alpha$}}}
\newcommand{\bbeta}{\mbox{\boldmath{$\beta$}}}
\newcommand{\bsigma}{\mbox{\boldmath{$\sigma$}}}
\newcommand{\btau}{\mbox{\boldmath{$\tau$}}}

\newcommand{\cQ}{\mathcal{Q}}
\newcommand{\cT}{\mathcal{T}}
\newcommand{\bx}{\mathbf{x}}

\newcommand{\bu}{\mathbf{u}}

\newcommand{\C}{\mathcal{C}}

\DeclareMathOperator*{\argmin}{arg\,min}

\newcommand{\sbtau}{\mbox{\boldmath{\scriptsize$\tau$}}}

\title{A New Minimisation Principle for Poisson Equation Leading 
to a Flexible Finite Element Approach}
\author{Bishnu P. Lamichhane
\thanks{School of Mathematical and Physical Sciences, University of Newcastle, 
Callaghan, NSW 2308,  {\tt Bishnu.Lamichhane@newcastle.edu.au}}}
\begin{document}
\maketitle
\begin{abstract}
We introduce  a  new minimisation principle 
for Poisson equation using two variables: the solution and the gradient 
of the solution. This principle allows us to 
use any conforming finite element spaces for both variables, where 
the finite element spaces do not need to satisfy a so-called 
inf-sup condition.  A numerical example demonstrates the 
superiority of the approach. 
 \end{abstract}


\section{Introduction}

It is often more important to get the accurate approximation of 
the gradient of the solution of a Poisson equation. In that case, 
a mixed formulation of the Poisson equation is used, where there 
are two unknowns - the solution and its gradient - in the variational equation. 
Discretising a mixed formulation of a partial differential equation is 
a challenging task as the involved finite element spaces should 
satisfy a compatibility condition - so called inf-sup condition. 
Although there are many finite 
element spaces discovered satisfying the compatibility condition for 
the Poisson equation \cite{BF91,BS94,Bra01,MTW02,BHM05,BD06,UDG06},
 it is not so easy for mixed formulations 
of other partial differential equations. It is sometimes 
useful to use a least-squares finite element method 
to approximate the solution and its gradient 
simultaneously \cite{BG98,BG05,BG09}. A least-squares formulation 
allows the use of any conforming finite element spaces avoiding 
the compatibility condition.

In this paper, we propose a new minimisation principle 
for Poisson equation using the solution and the 
gradient of the solution as 
two unknowns. 
This formulation is similar to 
a least-squares formulation in the sense that it allows the 
use of any conforming finite element spaces avoiding 
the compatibility condition \cite{Jia98,BG98,BG05,BG09}.
However,  in comparison to a least-squares finite element method, 
the source term $f$ can be in the dual of $H^1$-space, and the  gradient 
can be discretised using a $L^2$-conforming finite element space. 
We also prove optimal a priori error 
estimates for the proposed finite element method.

The structure of the rest of the paper is organised as follows. 
In the next section, we introduce our
formulation  and show its well-posedness. We propose 
 finite element methods for the given formulation and prove 
a priori error estimates in Section 3.
A numerical example with discretisation errors are presented in Section 4 and a short conclusion is drawn in Section 5.

\section{A New Formulation of Poisson equation}
In this section we introduce a new minimisation principle of the 
Poisson problem. Let  $\Omega \subset \BBR^d$, 
$ d\in \{2,3\}$,  be a bounded convex domain 
with polygonal or polyhedral boundary  $\partial\Omega$ with  
the outward pointing normal $\bn$ on  $\partial\Omega$.

 We start with the following minimisation problem for the Poisson problem 
 \begin{problem}\label{prob1}
 Given $f \in H^{-1}(\Omega)$ we want to find 
\begin{equation}\label{vbiharm}
u = \argmin\limits_{\substack{v \in H^1_0(\Omega)}}K(v)
\end{equation}
with 
\begin{equation}\label{func}
K(v)=\frac{1}{2} \int_{\Omega}|\nabla v|^2\,dx-\ell(v).
\end{equation}
where  \[ \ell(v) = \int_{\Omega}f\,v\,dx.\]
\end{problem}
We refer to the following references \cite{BF91,BS94,Bra01,MTW02,ABC03,BHM05,BD06,UDG06,
Ste10,Lam12a} for different variational formulations of Poisson equation. 

Let $V=H_0^1(\Omega)$ and $\bQ=[L^2(\Omega)]^d$,
and  for two vector-valued functions 
$\balpha:\Omega \rightarrow \BBR^d$ and
$\bbeta:\Omega \rightarrow \BBR^d$,
the Sobolev inner product on the Sobolev space 
$H^k(\Omega)$  ($ k \in \BBR$) be defined as 
\[ \la \balpha,\bbeta\ra_{k,\Omega}:={\sum_{i=1}^d \la 
\alpha_{i},\beta_{i}\ra_{k,\Omega}},
\] 
where $(\balpha)_{i}=\alpha_{i},\; (\bbeta)_{i}=\beta_{i}$
with $\alpha_{i}, \beta_{i} \in  H^k(\Omega)$, 
for $i=1,\cdots,d$,  and the norm $\|\cdot\|_{H^k(\Omega)}$ is 
induced from this inner product. We will use 
the standard notation 
$\|\cdot \|_{k,\Omega}$ for the norm in the $H^k(\Omega)$-space. 
We now introduce  a 
functional $J_{\alpha,\gamma}(v,\btau;f)$ with 
\[ J_{\alpha,\gamma} (v,\btau;f) =  \|\btau \|^2_{0,\Omega} + 
\|\btau - \alpha\nabla v\|^2_{0,\Omega} + \gamma \ell(v),\]
where $\alpha>0$ and $\gamma$ are two fixed constants, and 
consider another minimisation problem for two variables $(v,\btau) \in [V\times \bQ]$ 
\begin{equation}\label{mpoisson} 
\argmin\limits_{\substack{(v,\sbtau) \in [V\times \bQ]}} J_{\alpha,\gamma}(v,\btau;f).
\end{equation}

The minimisation problem is equivalent 
to finding  
$(u,\bsigma) \in [V\times \bQ]$ such that 
\begin{equation}\label{mvarm} 
a((u,\bsigma),(v,\btau)) = -\frac{\gamma}{2}\ell(v),\quad (v,\btau) \in [V\times \bQ],
\end{equation}
where the bilinear form $a(\cdot,\cdot)$ is defined as 
\[  a((u,\bsigma),(v,\btau))   = 
 (\bsigma, \btau)_{0,\Omega} 
+ (\bsigma - \alpha \nabla u, \btau -\alpha  \nabla v)_{0,\Omega}.\]

Standard arguments can be used to show the 
continuity of the bilinear form $a(\cdot,\cdot)$ on  
the space $V\times \bQ$. Now we show that 
the bilinear form $a(\cdot,\cdot)$ is coercive 
on $V\times \bQ$. 
\begin{lemma}
Let $\alpha$ and $\gamma$ be two constants with $\alpha>0$ .
For  $(u,\bsigma) \in [V\times \bQ]$ 
the bilinear from $a(\cdot,\cdot)$ satisfies 
\[ a((u,\bsigma),(u,\bsigma)) \geq  \frac{\alpha}{\alpha +2C_1} \left(
\|u\|^2_{1,\Omega} + \|\bsigma\|^2_{0,\Omega}\right),\]
where  $C_1$ is the constant in the 
 Poincar\'e inequality
\[ \|u\|^2_{1,\Omega}  \leq C_1 \|\nabla u\|_{0,\Omega}^2.\]
\end{lemma}
\begin{proof}
The proof follows from a triangle inequality and 
Poincar\'e inequality:
\begin{align*}
\|u\|^2_{1,\Omega} + \|\bsigma\|^2_{0,\Omega} 
&\leq \frac{C_1}{\alpha} \|\alpha \nabla u\|^2_{0,\Omega} + 
 \|\bsigma\|^2_{0,\Omega}  \\ & \leq 
 \frac{2C_1}{\alpha} \left[ \|\bsigma-\alpha\nabla u\|^2_{0,\Omega} + 
 \|\bsigma\|^2_{0,\Omega} \right] +  \|\bsigma\|^2_{0,\Omega}\\ &\leq
 \frac{2C_1+\alpha}{\alpha} \left( \|\bsigma\|^2_{0,\Omega} + \|\bsigma-\alpha\nabla u\|^2_{0,\Omega}\right)  \\ & = 
   \frac{2C_1+\alpha}{\alpha}  a((u,\bsigma),(u,\bsigma)).
\end{align*}
\end{proof}
\begin{corollary}
Since the bilinear form $a(\cdot,\cdot)$ is continuous and 
coercive on $V\times \bQ$, and the linear form $\ell(v)$ is also continuous on $V$ for $f\in H^{-1}(\Omega)$, the problem of  finding  
$(u,\bsigma) \in [V\times \bQ]$ such that 
\begin{equation}\label{mvarm} 
a((u,\bsigma),(v,\btau)) = -\frac{\gamma}{2}\ell(v),\quad (v,\btau) \in [V\times \bQ],
\end{equation}
has a unique solution from Lax-Milgram lemma. 
\end{corollary}
 \begin{remark}
 In contrast to the standard least-squares method, where we need to have $\ell \in L^{2}(\Omega)$, 
 we have here $f \in H^{-1}(\Omega)$.  Thus the standard least-squares method 
 cannot handle the situation if the source function is not $L^2$, whereas the new approach 
 requires exactly the same regularity for $f$ as the standard Galerkin approach.
 \end{remark}
Let $(u_e, \bsigma_e) \in V \times \bQ$ be the solution of  the minimisation problem 
\eqref{mpoisson}. 
We now choose $\alpha$ and $\gamma$ in such a way that 
the solution $(u,\bsigma)$  of the minimisation problem \eqref{mpoisson}.
satisfies $u=u_e$ and $ \bsigma_e = \nabla u$.
Here the natural norm for an element $(v, \btau) \in  V \times \bQ$ of the product space 
$V \times \bQ$ is $\sqrt{\|v\|^2_{1,\Omega} + \|\btau\|^2_{0,\Omega}}$. 
Thus  \eqref{mpoisson} leads to the problem of finding 
$(u,\bsigma) \in [V\times \bQ]$ such that 
\begin{equation*}
 2(\bsigma, \btau)_{0,\Omega} 
+ 2 (\bsigma - \alpha \nabla u, \btau - \alpha \nabla v)_{0,\Omega} 
+ \gamma \ell(v) =0, \quad (v,\btau) \in [V\times \bQ]. 
\end{equation*}
Letting the test functions $\btau =\bzero$  and $v=0$ successively 
in the above equation leads to 
\begin{equation}\label{mform}
\begin{array}{lccccc}
- 2 (\bsigma - \alpha \nabla u,  \alpha \nabla v)_{0,\Omega} 
+ \gamma \ell(v) & = & 0, \quad  v \in V, \\ 
 (\bsigma, \btau)_{0,\Omega}  + 
 (\bsigma - \alpha \nabla u, \btau )_{0,\Omega}  & = & 0,\quad 
\btau \in \bQ.
\end{array}
\end{equation}
The second equation immediately yields 
 \[ (2\bsigma- \alpha \nabla u, \btau)_{0,\Omega}   =  0,\quad 
\btau \in \bQ,\]
and hence $\alpha =2$ ensures that $\bsigma = \nabla u$.  
Using $\bsigma = \nabla u$ in the first equation of \eqref{mform}, 
we have 
\[ -2 \alpha (1-\alpha) (\nabla u, \nabla v)_{0,\Omega}  + \gamma \ell(v) =0.\]
We have the standard variational problem for the Poisson equation 
if $ \gamma = 2 \alpha (1-\alpha)$, and thus setting $\alpha =2$, we get 
$\gamma = -4$.  
Now we have the following problem. 
\begin{problem} \label{prob2}
Given $f \in H^{-1}(\Omega)$, 
the variational equation for the  
minimisation problem is to find 
$(u,\bsigma) \in [V\times \bQ]$ such that 
\begin{equation}\label{mvarm} 
a((u,\bsigma),(v,\btau)) = 2\ell(v),\quad (v,\btau) \in [V\times \bQ],
\end{equation}
where the bilinear form $a(\cdot,\cdot)$ is defined as 
\[  a((u,\bsigma),(v,\btau))   = 
 (\bsigma, \btau)_{0,\Omega} 
+ (\bsigma - 2\nabla u, \btau -2 \nabla v)_{0,\Omega}.\]
 \end{problem}
 From the above discussion we have the following proposition. 
 \begin{proposition}
 Let $u$ be the solution of  Problem \ref{prob1} and $(\tilde u,\tilde \bsigma)$ of Problem \ref{prob2}. 
 Then we have  $\tilde u = u$ and $\tilde \bsigma = \nabla u$. 
 \end{proposition}

 \begin{remark}
 The idea can be easily generalised to a general differential equation, which can be put in 
 a minimisation framework. For example, consider 
 the solution of the linear elastic problem of finding the displacement field $\bu \in \bV = [H^1_0(\Omega)]^d$ such that 
 \cite{Bra01} 
 \[ \bu = \argmin_{\bv \in \bV }\frac{1}{2}  \int_{\Omega} \varepsilon(\bv) : \C \varepsilon(\bv) \, dx - \ell(\bv),\]
 where $\varepsilon(\bv) = \frac{1}{2} \left( \nabla \bv + [\nabla \bv]^T\right)$ 
 is the symmetric part of the gradient,  $\C$ is the Hooke's 
 tensor, and $\ell(\cdot)$ is a linear form 
 \[ \ell(\bv) = \int_{\Omega} \fb \cdot \bv \, dx.\]
 By defining a pseudo-stress $\bsigma = \sqrt{\C} \varepsilon(\bv)$, we can put this in the above framework with 
 \[ a(\bu, \bsigma, \bv, \btau) = (\bsigma, \btau)_{0,\Omega} + (\bsigma - 2\sqrt{\C} \varepsilon(\bu), \btau - 2\sqrt{\C} \varepsilon(\bv))_{0,\Omega}.\]
 We note that since $\C$ is a symmetric positive definite tensor, its square is well-defined.
  \end{remark}

\section{Finite element approximation and a priori error estimate}\label{fe}
Let $\cT_h$ be a quasi-uniform partition of the domain  $\Omega$ in 
simplices, convex quadrilaterals or hexahedra having the mesh-size $h$. 
Let $\hat{T}$ be a reference simplex or square or cube, where the 
reference simplex is defined as 
\[ \hat T:= \{ \bx \in \BBR^d:\,x_i>0,\;i=1,\cdots,d,\;\text{and}\;\sum_{i=1}^dx_i <1\},\] 
and the reference square or cube $\hat{T}:=(0,1)^d$.
 
The finite element space is defined by affine maps $F_T$ from a 
reference element $\hat{T}$ to  
a physical element $T\in \cT_h$. 
For $k \geq 0$, let $\cQ_k(\hat T)$ be the space of polynomials 
of degree less than or equal to $k$ in $\hat T$ in the variables 
$x_1,\cdots, x_d$ if $\hat T$ is the reference simplex,  the space of 
polynomials in $\hat T$ of degree less than or equal to $k$ with 
respect to each variable $x_1,\cdots, x_d$ if $\hat T$ is  
the reference square or cube. 

Then the finite element space  based on the mesh $\cT_h$ 
is defined as  the space of continuous functions whose restrictions 
to an element $T$ are obtained by maps of given 
polynomial functions from the reference element; that is,
\begin{equation*}\label{fespace}
S_h:=\left \{v_h \in H^1(\Omega):\,
v_h|_T=\hat{v}_h\circ F^{-1}_{T},\ \ \hat{v}_h\in
\cQ_k(\hat{T}),\;T\in \cT_h\right \},
\end{equation*}
  see \cite{Cia78,QV94, BS94,Bra01}. We now define 
 $V_h : = S_h \cap H^1_0(\Omega)$.
We also want to define two other finite element
 spaces $L_h$ and $\bQ_h$ as 
\begin{equation*}\label{fespacen}
L_h:=\left \{v_h \in L^2(\Omega):\,
v_h|_T \in \cQ_k(T),\;T\in \cT_h\right \},\quad 
\bQ_h : = [L_h]^d,
\end{equation*}

Now a discrete formulation of our problem is to find 
$(u_h,\bsigma_h) \in [V_h\times \bQ_h]$ such that 
\begin{equation}\label{mvar} 
a((u_h,\bsigma_h),(v_h,\btau_h)) = 2\ell(v_h),\quad (v_h,\btau_h) \in [V_h\times \bQ_h].
\end{equation}
Since $[V_h]^d \subset \bQ$, we can also use $\bV_h = [V_h]^d$ to 
discretize the  gradient of the continuous problem. 
 This leads to a problem of finding 
$(u_h,\bsigma_h) \in [V_h\times \bV_h]$ such that 
\begin{equation}\label{mvarn} 
a((u_h,\bsigma_h),(v_h,\btau_h)) = 2\ell(v_h),\quad (v_h,\btau_h) \in [V_h\times 
\bV_h],
\end{equation}
which utilizes equal order interpolation. Since the discrete formulation 
is conforming, the bilinear form $a(\cdot,\cdot)$ and the 
linear form $\ell(\cdot)$ are both continuous on 
the corresponding spaces. 
The coercivity also follows from the continuous setting. 
\begin{theorem}
Thus the discrete problem of finding 
$(u_h,\bsigma_h) \in [V_h\times \bQ_h]$  or 
$(u_h,\bsigma_h) \in [V_h\times \bV_h]$ such that 
\[ a((u_h,\bsigma_h),(v_h,\btau_h)) = 2\ell(v_h),\quad (v_h,\btau_h) \in [V_h\times 
\bQ_h]\quad\text{or}\quad 
(v_h,\btau_h) \in [V_h\times 
\bV_h]\]
has a unique solution, and the solution satisfies 
\[ \|u-u_h\|_{1,\Omega} + 
\|\bsigma - \bsigma_h\|_{0,\Omega} 
\leq c \left(\inf_{v_h\in V_h} \|u-v_h\|_{1,\Omega}
+ \inf_{\btau_h\in \bV_h\;\text{or}\; \btau_h \in \bQ_h} \|\bsigma-\btau_h\|_{0,\Omega}\right),\]
where $u$ is the exact solution of the 
problem \eqref{vbiharm} and 
$\bsigma=\nabla u$. 
\end{theorem}
\begin{proof}
The proof follows from Galerkin orthogonality and 
standard arguments.
\end{proof}
\begin{remark}
The solution $u$ is assumed to be in $H^1_0(\Omega)$ only for the purpose of simplicity. 
In fact, any non-zero Dirichlet condition or mixture of Dirichlet and Neumann boundary 
conditions are all fine as in the case of the standard Galerkin finite element method. 
\end{remark}
\section{Numerical example}
In this section we consider a numerical example to demonstrate the performance of 
this new minimisation scheme. In fact, we show the discretisation errors for 
the solution $u$ in the $L^2$ and $H^1$-norms, and discretisation errors for 
the gradient in the $L^2$-norm. For this example, we consider the domain 
of the square $\Omega = [-1,1]^2$ with the exact solution 
\[ u (x,y) = ((x-y)\exp(-5.0(x-0.5)(x-0.5)-5.0(y-0.5)(y-0.5))),\]
where the right-hand side function $f$ and the Dirichlet boundary condition on $\partial \Omega$ is 
obtained by using this exact solution. The two components of the gradient are denoted by 
 $\sigma^1$ and $\sigma^2$, respectively, where their numerical approximations are denoted by 
   $\sigma_h^1$ and $\sigma_h^2$, respectively. 
We start with the initial uniform triangulation of 32 triangles in the first level and then refine 
uniformly in each level. We have tabulated the discretisation errors in Table \ref{exa}. 
We can see that the numerical results are as predicted by the theory. 
Moreover, the discretisation errors show the superiority of the 
scheme as the discretisation errors for the gradient of the solution converge quadratically to the exact solution. 
This is normally not achieved in any mixed finite element methods.  

\begin{table}[!htb]
\caption{Discretisation errors for the solution and gradient}
\begin{center}
\begin{tabular}{|c|c|c|c|c|c|c|c|c|c|}\hline
$l$ & \multicolumn{2}{c|}{$\|u -u_h\|_{1,\Omega}$} & \multicolumn{2}{c|}{$\|u -u_h\|_{0,\Omega}$}
 & \multicolumn{2}{c|}{$\|\sigma^1 -\sigma_h^1\|_{0,\Omega}$} & \multicolumn{2}{c|}{$\|\sigma^2 -\sigma_h^2\|_{0,\Omega}$} \\ \hline
  1 &    3.41208e-01 &  & 3.71839e-02 &  & 1.72948e-01 &  & 1.72948e-01 &  \\\hline
 2 &   1.70261e-01 & 1.00 & 1.15857e-02 & 1.68 & 5.39510e-02 & 1.68 & 5.39510e-02 & 1.68  \\\hline
 3 &    8.40661e-02 & 1.02 & 3.09011e-03 & 1.91 & 1.45061e-02 & 1.89 & 1.45061e-02 & 1.89  \\\hline
 4 &  4.18485e-02 & 1.01 & 7.86293e-04 & 1.97 & 3.74056e-03 & 1.96 & 3.74056e-03 & 1.96  \\\hline
 5 &  2.08998e-02 & 1.00 & 1.97503e-04 & 1.99 & 9.62306e-04 & 1.96 & 9.62306e-04 & 1.96  \\\hline
 6 & 1.04469e-02 & 1.00 & 4.94393e-05 & 2.00 & 2.51753e-04 & 1.93 & 2.51753e-04 & 1.93  \\\hline
\end{tabular}
\end{center}
\label{exa}
\end{table}

\section{Conclusion}
We have proposed a new minimisation principle 
for the Poisson equation based on  the solution and its  gradient. 
One big advantage of this formulation is that 
a finite element approximation can be performed as in a 
least-squares finite element method without 
fulfilling the compatibility condition between two 
finite element spaces. However, the finite element 
approach is much easier than in a least-squares approach. 
An optimal a priori error estimate is proved for the 
proposed formulation. A numerical example is presented to 
demonstrate the optimality of the scheme. 

\section*{Acknowledgement}
Support from the near miss grant of the University of Newcastle 
 is gratefully acknowledged.

\bibliographystyle{plain}
\bibliography{total}

\begin{thebibliography}{10}

\bibitem{ABC03}
Y.~Achdou, C.~Bernardi, and F.~Coquel.
\newblock A priori and a posteriori analysis of finite volume discretizations
  of {D}arcy's equations.
\newblock {\em Numerische Mathematik}, 96:17--42, 2003.

\bibitem{BD06}
P.B. Bochev and C.R. Dohrmann.
\newblock A computational study of stabilized, low-order ${C}^0$ finite element
  approximations of {D}arcy equations.
\newblock {\em Computational Mechanics}, 38:323--333, 2006.

\bibitem{BG98}
P.B. Bochev and M.~Gunzburger.
\newblock Finite element methods of least-squares type.
\newblock {\em SIAM Review}, 40:789--837, 1998.

\bibitem{BG05}
P.B. Bochev and M.~Gunzburger.
\newblock On least-squares finite element methods for the poisson equation and
  their connection to the {D}irichlet and {K}elvin principles.
\newblock {\em SIAM Journal on Numerical Analysis}, 43:340--362, 2005.

\bibitem{BG09}
P.B. Bochev and M.~Gunzburger.
\newblock {\em Least-Squares Finite Element Methods}.
\newblock Springer, New York, 2009.

\bibitem{Bra01}
D.~Braess.
\newblock {\em Finite Elements. Theory, Fast Solver, and Applications in Solid
  Mechanics}.
\newblock Cambridge Univ. Press, Second Edition, Cambridges, 2001.

\bibitem{BS94}
S.C. Brenner and L.R. Scott.
\newblock {\em The Mathematical Theory of Finite Element Methods}.
\newblock Springer--Verlag, New York, 1994.

\bibitem{BF91}
F.~Brezzi and M.~Fortin.
\newblock {\em Mixed and hybrid finite element methods}.
\newblock Springer--Verlag, New York, 1991.

\bibitem{BHM05}
F.~Brezzi, T.J.R. Hughes, L.D. Marini, and A.~Masud.
\newblock Mixed discontinuous galerkin methods for {D}arcy flow.
\newblock {\em Journal of Scientific Computing}, 22-23:119--145, 2005.

\bibitem{Cia78}
P.G Ciarlet.
\newblock {\em The Finite Element Method for Elliptic Problems}.
\newblock North Holland, Amsterdam, 1978.

\bibitem{Jia98}
B.~Jiang.
\newblock {\em The Least-Squares Finite Element Method: Theory and Applications
  in Computational Fluid Dynamics and Electromagnetics}.
\newblock Springer--Verlag, 1998.

\bibitem{Lam12a}
B.P. Lamichhane.
\newblock Mixed finite element methods for the poisson equation using
  biorthogonal and quasi-biorthogonal systems.
\newblock {\em Advances in Numerical Analysis}, 2013:9 pages, 2013.

\bibitem{MTW02}
K.A. Mardal, X.-C. Tai, and R.~Winther.
\newblock A robust finite element method for {D}arcy--{S}tokes flow.
\newblock {\em SIAM Journal on Numerical Analysis}, 40:1605--1631, 2002.

\bibitem{QV94}
A.~Quarteroni and A.~Valli.
\newblock {\em Numerical approximation of partial differential equations}.
\newblock Springer--Verlag, Berlin, 1994.

\bibitem{Ste10}
R.~Stenberg.
\newblock A nonstandard mixed finite element family.
\newblock {\em Numerische Mathematik}, 115:131--139, 2010.

\bibitem{UDG06}
J.M. Urquiza, D.N'. Dri, A.~Garon, and M.C. Delfour.
\newblock A numerical study of primal mixed finite element approximations of
  {D}arcy equations.
\newblock {\em Communications in Numerical Methods in Engineering},
  22:901--915, 2006.

\end{thebibliography}
\end{document}